\documentclass[pdflatex,sn-mathphys-num]{sn-jnl}

\usepackage{graphicx}%
\usepackage{multirow}%
\usepackage{amsmath,amssymb,amsfonts}%
\usepackage{amsthm}%
\usepackage{mathrsfs}%
\usepackage[title]{appendix}%
\usepackage{xcolor}%
\usepackage{textcomp}%
\usepackage{manyfoot}%
\usepackage{booktabs}%
\usepackage{algorithm}%
\usepackage{algorithmicx}%
\usepackage{algpseudocode}%
\usepackage{listings}%
\usepackage{array}
\usepackage{rotating}
\usepackage{geometry} 
\usepackage{subfigure}
\newtheorem{theorem}{Theorem}% 
\newtheorem{proposition}{Proposition}% 

\newtheorem{lemma}{Lemma}
\newtheorem*{lemma star}{Lemma}

\renewcommand{\overline}{\bar}
\renewcommand{\thefigure}{\arabic{figure}}
\begin{document}
\title[Article Title]{\bf Convergence of the Immersed Interface Method in Linear Elasticity}

\author*[1,2]{\fnm{Sabia} \sur{Asghar}}\email{sabia.asghar[AT]uhasselt.be}

\author*[3,4]{\fnm{Qiyao} \sur{Peng}}\email{qiyao.peng[AT]lancaster.ac.uk}

\author[5]{\fnm{Etelvina} \sur{Javierre}}\email{etelvina[AT]unizar.es}
\equalcont{These authors contributed equally to this work.}

\author[1,2,6,7]{\fnm{Fred} \sur{Vermolen}}\email{fred.vermolen[AT]uhasselt.be}
\equalcont{These authors contributed equally to this work.}

\affil[1]{\orgdiv{Department of Mathematics and Statistics, Computational Mathematics Group}, \orgname{University of Hasselt}, \orgaddress{\street{Diepenbeek}, \city{Hasselt}, \postcode{3590}, \country{Belgium}}}

\affil[2]{\orgname{Data Science Institute (DSI), University of Hasselt}, \orgaddress{\street{Diepenbeek}, \city{Hasselt}, \postcode{3590}, \country{Belgium}}}

\affil[3]{\orgdiv{Mathematics for AI in Real-world Systems, School of Mathematical Sciences}, \orgname{Lancaster University}, \orgaddress{ \city{Lancaster}, \postcode{LA1 4YF}, \country{UK}}}

\affil[4]{\orgdiv{Mathematical Institute}, \orgname{ Leiden University}, \orgaddress{ \city{Leiden}, \postcode{Einsteinweg 55 2333 CC}, \country{The Netherlands}}}

\affil[5]{\orgdiv{IUMA \& Applied Mathematics Department}, \orgname{University of Zaragoza}, \orgaddress{\city{Zaragoza}, \postcode{50009}, \country{Spain}}}

\affil[6]{\orgdiv{Department of Mathematics and Applied Mathematics}, \orgname{University of Johannesburg}, \orgaddress{\city{Johannesburg}, \postcode{2006}, \country{South Africa}}}

\affil[7]{\orgdiv{Delft Institute of Applied Mathematics}, \orgname{Delft University of Technology}, \orgaddress{\street{Mekelweg 4}, \city{Delft}, \postcode{2628 CD Delft}, \country{The Netherlands}}}

\abstract{We consider an open, bounded, simply connected (Lipschitz) domain in $\mathbb{R}^d$, which contains a closed polyhedral surface or polygonal contour, referred to as the interface. From this interface, forces are exerted in the normal direction. The forces are continuously distributed over the interface, resulting in an integral expression. This features an important characteristic of the immersed interface method. Since the integral cannot be resolved exactly, one relies on numerical quadrature rules to approximate the integral. Therefore, we consider two different linear elasticity problems with forces over a curve or surface (interface) that is located within the (open) domain of computation:
(1) The force is defined by an integral over the interface; (2) The force is defined by a quadrature approximation of the integral over the interface. We prove that the L2–norm of the difference between the solutions from the two elasticity problems is of the same order as the error of quadrature. The results are demonstrated for both bounded and unbounded domains. The proof that we establish relies on the use of: (i) fundamental solutions for linear elasticity, exhibiting singular behaviors (in particular around points of action) and not being in ${\bf H}^1$, and (ii) on the use of singularity removal principle and the Extended Trace Theorem. Convergence is demonstrated in the ${\bf L}^2$--norm on curves and manifolds. We show some numerical experiments on the basis of fundamental solutions with a Midpoint quadrature rule in an unbounded and a bounded domain.  We note that the error that we estimate is for the exact solutions and not for finite element solutions. Hence in the numerical finite element-based simulations, the numerical results contain an additional error due to the finite element approach.}


\maketitle

\section{Introduction}
\label{sec: Introduction}
In many biomedical processes such as wound healing, tumor growth, organ development or metastasis of cancer, cell migration is a fundamental process. Whenever a cell needs to migrate, it adheres to the substrate and applies pulling forces to 'crawl' optimally \citep{DePascalis2017, Wang2023, Zheng2019}. These pulling forces are then transmitted to the extracellular matrix (ECM) and remodel the ECM \cite{Totsukawa2004, Schwartz2010, Lecuit2011}. The way cellular forces impact the ECM, often causes significant biological implications: as an example, we mention dermal contractures that are caused by excessive forces exerted by fibroblasts and myofibroblasts \citep{Bin_2011} after a severe burn injury or other deep tissue injury; in the context of cancer, cellular forces facilitate and promote the migration and invasion of cancer cells to other body parts through remodeled ECM \citep{Daphne_2013, Daphne_2021, DePascalis2017}. 

There are various mathematical frameworks to model how the cellular pulling forces deform the ECM \citep{Zheng2019, Peng_2022, Peng_2022_JCAM,ArellanoTint2025}. 
A physically accurate way to model the interaction of cells and their surroundings (tissue) is to define the intra- and extracellular regions as separate subdomains and to implement the forces by a inhomogeneous Neumann boundary condition on the cell membrane \citep{Peng_2022_JCAM}. However, when cells migrate, the mesh structure needs to be regenerated at every time step, increasing the computational workload, which is not ideal particularly for a large number of cells. Therefore, a more computationally efficient model is to utilize the Immersed Interface Approach \citep{Yuri_2025, roy2020immersed} (also known as Immersed Boundary Method), where the forces are implemented by (migrating) point forces in the form of a Dirac delta distribution at the cell membrane. This approach also has the flexibility of handling more complicated and migrating geometries \citep{Peng2020-lu, Peng_Cancer_2023,alvarez2024}.

Linear elasticity is one of the fundamental models for describing how the material deforms in response to the forces applied to it.
Of course, we are aware of the crude and inaccurate nature of the approximation that Hooke's Law for linear elasticity provides for tissue mechanics. Large deformations from cellular and external forces typically require more accurate hyperelastic, viscoelastic, or morphoelastic models.  
However, this paper \textit{does not} aim to develop the most precise method for computing displacement fields. Instead, we focus on demonstrating and analyzing the accuracy of the Immersed Interface Method in its simplest mechanical context (i.e. linear elasticity model without inertia) before investigating more complicated models. The advantage is that one can utilize the superposition principle and the fundamental solution.

In \citet{Peng_2022_JCAM}, the authors have shown the transition between the so-called 'hole' approach (the cellular force modeled as a smooth function in the inhomogeneous Neumann boundary condition over the cell membrane) and the 'immersed interface approach' (the cellular force modeled as an integral of the Dirac delta distributions over the cell membrane), for the case that the displacement field was obtained by the use of the finite element method (FEM). While the discretization error from the FEM applied to the boundary value problem (BVP) is well-known, the numerical implementation of the immersed interface method introduces an additional error as a result of the quadrature method that is needed to approximate the integral expression over the interface for the forcing.

In this study, our aim is to investigate how this quadrature error will influence the accuracy of the solutions to the immersed boundary approach, where the cellular forces are expressed as Dirac delta distributions, in the context of linear elasticity. Theoretically, the forces over the interface are in the form of an integral, whereas in the context of numerical simulation, this integral becomes a summation of terms that results from the choice of the quadrature rule. Unlike the approach used in \citet{Peng_2022_JCAM}, the analysis conducted in the current study utilizes the fundamental solution of linear elasticity in the fields $\mathbb{R}^2$ or $\mathbb{R}^3$, and is valid for the solution to the continuous problem. 

This study shares some similarities with the Boundary Element Method (BEM) \citep{Brebbia1992, Kythe1995, Alarcon1972} in its reliance on fundamental solutions (Green's functions) for linear elasticity problems. However, there are several important distinctions: unlike BEM which uses a variational formulation with boundary integrals (Somigliana's identity \citep{Friedel1981}) to enforce boundary conditions and requires solving a system of linear equations with a dense coefficient matrix, which makes it tedious from an algebraic point of view. Our approach does not require (iterative) adjustment of solutions such that boundary conditions are satisfied. Hence our formulation paves the road for efficient computation.

The novel aspects of our work include the application of forces on an internal closed contour or surface within a bounded domain, the use of fundamental solutions for steady-state linear elasticity with infinitesimal strain, and the use of the extended trace theorem to establish convergence rates in regions away from where forces are applied. As mentioned earlier, these features allow our method to bypass some of the computational complexities inherent in traditional BEM approaches \citep{Kythe1995, Alarcon1972} while maintaining rigorous convergence properties, as demonstrated through our analytical results.

The manuscript is structured as follows. The mathematical formulation is presented in Section \ref{sec:2}. In Section \ref{sec:3}, we show the convergence between the solutions to the linear elasticity model, which use the integral and summation forms for the cellular forces in any dimensionality larger than one. Section \ref{sec:4} presents the numerical results and finally we deliver the conclusions in Section \ref{sec:5}.

\section{The Model Equations and the Fundamental Solutions}
\label{sec:2}

Starting from the immersed interface approach in \citet{Peng_2022_JCAM}, we consider the linear elasticity equation with cellular forces, expressed as a collection of Dirac delta distributions, both in the form of an integral and a summation. By virtue of the availability of Green's function, in this section, we treat two- and three-dimensional cases.

\subsection{The Boundary Value Problems}
We consider a non-empty, open, bounded, simply connected Lipschitz domain $\Omega \in \mathbb{R}^d$, bounded by $\partial\Omega$ (hence, $\overline{\Omega} = \Omega \cup \partial \Omega$) for $d \in \{2,3\}$. In $\Omega$, there is a non-degenerate closed curve or surface (begin and end points coincide) that is denoted by $\Gamma$, of which the non-empty enclosed region is regarded as a cell (or tumor) in the context of biological application. The forces are applied on $\Gamma$ in the normal direction towards or away from the cell center. We note that the direction of the cellular forces can be reverted or generalized easily. For the purpose of evaluating the solutions (which will be elaborated in Section \ref{sec:3}), we define a larger circle (or sphere in three dimensions) that encloses $\Gamma$. This circle (or sphere) is denoted by $S$ and we evaluate the displacement over $S$. Figure \ref{fig1} presents a two-dimensional schematic with a circular cell, circular $S$ and square $\Omega$. The sketch can also represent a side view for the three-dimensional case, with a spherical cell, spherical $S$ and a cubic domain of computation $\Omega$.

\begin{figure}
    \centering
    \includegraphics[width=0.85\linewidth]{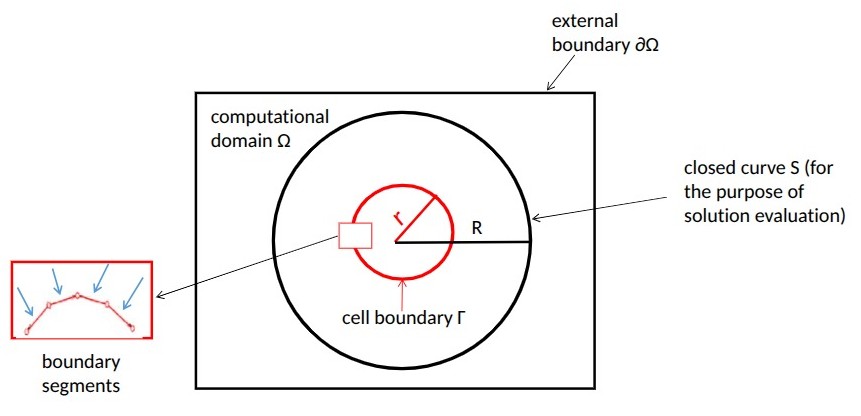}
    \caption{\it A schematic of two-dimensional case is shown, in which $\Omega\subset\mathbb{R}^2$ is an open bounded square domain with boundary $\partial\Omega$; $\Gamma$ represents the boundary of a (circular) cell (or tumor), which is split into finite number of line segments and point forces are applied at the midpoint of these line segments; $S$ is a circular curve strictly embedded in $\Omega$ and strictly enclosing $\Gamma$. We evaluate the ${\bf L}^2$--norm of the solution over $S$. Note that this schematic can also be regarded as a side view of the three-dimensional setting with cube and spheres, respectively.}
    \label{fig1}
\end{figure}

Let ${\boldsymbol x} \in \Omega \subset \mathbb{R}^d$ denote the spatial coordinate, where $\Omega$ is the computational domain. The displacement field is denoted by ${\boldsymbol u}:{\boldsymbol x} \longrightarrow \boldsymbol u(\boldsymbol x)$, ${\boldsymbol u} :\Omega \longrightarrow \mathbb{R}^d$, representing the displacement of a material from its reference configuration. In the current study, inertial effects are neglected and linear elasticity is considered, which gives 
 $$-\nabla \cdot \boldsymbol\sigma(\boldsymbol{u}) = \boldsymbol{f},$$ with $\boldsymbol{f}$ representing the body force. Here, the stress-strain relation is defined by Hooke's law:
\begin{equation*}
    \left\{
        \begin{aligned}
            \boldsymbol\sigma &= \frac{E}{1+\nu }\left\{ \epsilon +  \frac{\nu }{1-2\nu } tr(\epsilon )\mathbf{I}\right\},\\
            \boldsymbol\epsilon &= \frac{1}{2}\left[\nabla\boldsymbol{u}+ \nabla\boldsymbol{u}^T\right],
        \end{aligned}
    \right.
\end{equation*}
where $\boldsymbol \sigma$ and $\boldsymbol \epsilon$, represent the stress tensor and infinitesimal strain tensor, respectively; $E$ denotes the Young's Modulus of the material (cell and immediate surroundings), $\nu $ is the Poisson's ratio.

For the computational domain $\Omega$ and together with the homogeneous Dirichlet boundary condition, the two BVPs with the cellular forces in the integral and summation forms are given by
\begin{equation*}
    \mathrm{(BVP_{int})}\left\{
    \begin{aligned}
        - \nabla \cdot {\boldsymbol \sigma}({\boldsymbol u}) &= \int_{\Gamma}Q(\boldsymbol{x}^{\prime })\boldsymbol{n}(\boldsymbol{x}^{\prime })\delta (\boldsymbol{x}-\boldsymbol{x}^{\prime })d\Gamma (\boldsymbol{x}^{\prime}), &\mbox{$\boldsymbol{x}\in \Omega$},\\
        \boldsymbol{u}(\boldsymbol{x})&=\boldsymbol{0}, &\mbox{$\boldsymbol{x}\in\partial\Omega$,}
    \end{aligned}
    \right.
\end{equation*}
and 
\begin{equation*}
    \mathrm{(BVP_{sum})}\left\{
    \begin{aligned}
        - \nabla \cdot {\boldsymbol\sigma}({\boldsymbol u}_h) &= \sum\limits_{j=1}^{m} w_jQ(\boldsymbol{x}_{j})\boldsymbol{n}(\boldsymbol{x} _{j})\delta (\boldsymbol{x}-\boldsymbol{x}_{j}), &\mbox{$\boldsymbol{x}\in \Omega$},\\
        \boldsymbol{u_h}(\boldsymbol{x})&=\boldsymbol{0}, &\mbox{$\boldsymbol{x}\in\partial\Omega$,}
    \end{aligned}
    \right.
\end{equation*}
where $Q(\boldsymbol{x})$ is the force magnitude, $\boldsymbol{n}(\boldsymbol{x})$ is the unit norm vector on $\Gamma$ towards the cell center, $\delta(\boldsymbol{x})$ is the Dirac delta distribution, $m$ is the number of discretized segments of $\Gamma$ (as shown in Figure \ref{fig1}) for the application of the quadrature rule to the numerical simulation and $w_j$ is the weight resulting from the selected quadrature rule.

\subsection{Fundamental Solutions and Convolution-Based Solutions}
\noindent
For a single point force in an infinite domain in the field $\mathbb{R}^d$ for $d = 2,3$ with force $\boldsymbol{F}$, the linear elasticity equation reads as
\begin{equation*}
-\nabla \cdot \boldsymbol\sigma({\boldsymbol u}) ={\boldsymbol {F}}~\delta ({\boldsymbol{x}}-{\boldsymbol{x}}%
^{\prime }), \qquad {\boldsymbol x} \in \mathbb{R}^d.   \label{8}
\end{equation*}%
The solution to this equation in $\mathbb{R}^d$ is given by
\begin{align*}
\label{Eq:fundamental_sol}
{\boldsymbol u}({\boldsymbol x}) = {\boldsymbol G}({\boldsymbol x},{\boldsymbol x'}) {\boldsymbol F}, \qquad {\boldsymbol x} \in \mathbb{R}^d,
\end{align*}
where $\boldsymbol{G}$ denotes the Green's tensor given by 
\cite{Weinberger_2005}
\begin{equation}
\boldsymbol{G}(\boldsymbol{x},\boldsymbol{x}^{\prime })=\left\{ 
\begin{array}{ll}
\displaystyle{\ \frac{1}{8\pi \mu (1-\nu )}\left( -(3-4\nu )\log (||{%
\boldsymbol{x}}-{\boldsymbol{x}}^{\prime }||){\boldsymbol{I_2}}+\frac{({%
\boldsymbol{x}}-{\boldsymbol{x}}^{\prime })\otimes ({\boldsymbol{x}}-{%
\boldsymbol{x}}^{\prime })}{||{\boldsymbol{x}}-{\boldsymbol{x}}^{\prime
}||^{2}}\right) }, & \text{in }\mathbb{R}^{2} \\ 
&  \\ 
\displaystyle{\frac{1+\nu }{16\pi E(1-\nu )}\left[ \frac{(3-4\nu )%
\boldsymbol{I}_{3}}{||\boldsymbol{x}-\boldsymbol{x}^{\prime }||}+\frac{(%
\boldsymbol{x}-\boldsymbol{x}^{\prime })\otimes (\boldsymbol{x}-\boldsymbol{x%
}^{\prime })}{||\boldsymbol{x}-\boldsymbol{x}^{\prime }||^{3}}\right] }, & 
\text{in }\mathbb{R}^{3}%
\end{array}%
\right.  \label{10}
\end{equation}%
with the Lam\'{e} coefficients $\mu $ and $\lambda $ are related to the Young's Modulus and Poisson ratio through the following equations: 
$$\mu =\frac{E}{2(1+\nu )}\text{ and }\lambda =\frac{E\nu }{(1+\nu )(1-2\nu )}.$$

In $\mathbb{R}^d$, we consider the linear elasticity model with the forces in $\mathrm{(BVP_{int})}$ and $\mathrm{(BVP_{sum})}$, respectively:
\begin{equation}
-\nabla \cdot \boldsymbol{\sigma }(\boldsymbol{u}^{\ast })=\int_{\Gamma
}Q(\boldsymbol{x}^{\prime })\boldsymbol{n}(\boldsymbol{x}^{\prime })\delta (%
\boldsymbol{x}-\boldsymbol{x}^{\prime })d\Gamma (\boldsymbol{x}^{\prime
}), \quad \boldsymbol{x}\in\mathbb{R}^d,  \label{eq:linear_elas_u_star}
\end{equation}
and 
\begin{equation}
-\nabla \cdot \boldsymbol{\sigma }(\boldsymbol{u}_{h}^{\ast
})=\sum\limits_{j=1}^{m} w_jQ(\boldsymbol{x}_{j})\boldsymbol{n}(\boldsymbol{x} _{j})\delta (\boldsymbol{x}-\boldsymbol{x}_{j}), \quad \boldsymbol{x}\in\mathbb{R}^d.  \label{eq:linear_elas_u_h_star}
\end{equation}

Solving the above equations with applying the Superposition Principle, we obtain the following theoretical 'field results':
\begin{equation}
\boldsymbol{u}^{\ast }(\boldsymbol{x})=\int_{\Gamma }Q(\boldsymbol{x}%
^{\prime })\boldsymbol{G}(\boldsymbol{x},\boldsymbol{x}^{\prime })%
\boldsymbol{n}(\boldsymbol{x}^{\prime })d\Gamma (\boldsymbol{x}^{\prime
}), \qquad {\boldsymbol x} \in \mathbb{R}^d,  \label{eq:u_star}
\end{equation}
and 
\begin{equation}
{\boldsymbol{u}}_{h}^{\ast }({\boldsymbol{x}})=\sum\limits_{j=1}^{m} w_j Q(%
\boldsymbol{x}_{j})\boldsymbol{G}(\boldsymbol{x},\boldsymbol{x}_{j})%
\boldsymbol{n}(\boldsymbol{x}_{j}), \qquad {\boldsymbol x} \in \mathbb{R}^d.  \label{eq:u_h_star}
\end{equation}%
The Green's function does not lie in ${\bf H}^1(\mathbb{R}^d)$, hence, the solutions defined above are not in ${\bf H}^1(\mathbb{R}^d)$ either for $d = 2,3$. We note that for any subset of $\mathbb{R}^d$ excluding the location of the point force, the solution is in ${\bf H}^1$.

\section{Convergence of the BVPs with Different Forms of Forces}
\label{sec:3}
In this section, we investigate the convergence between the solutions to $\mathrm{(BVP_{int})}$ and $\mathrm{(BVP_{sum})}$. First, we introduce some notation for Hilbert spaces of vector-fields. The space ${\bf L}^2(\Omega)$ for vector-fields is defined by
\[
{\bf L}^2(\Omega) := \left\{ {\boldsymbol v}: \Omega \longrightarrow \mathbb{R}^d: \int_{\Omega} ||{\boldsymbol v}||^2 d \Omega < \infty\right\}, 
\]
with its associated norm
\[
||{\boldsymbol v}||^2_{{\bf L}^2(\Omega)} = \int_{\Omega} || {\boldsymbol v} ||^2 d \Omega.
\]
In some instances, we also need the ${\bf L}^2$-norm of the vector-field over the (outer) boundary, given by
\[
|| {\boldsymbol v} ||_{{\bf L}^2(\partial \Omega)} =
\int_{\partial \Omega} || {\boldsymbol v} ||^2 d \Omega.
\]
Since we also deal with tensors, we introduce the ${\bf L}^2$-norm of a tensor function $\boldsymbol V$, ${\boldsymbol V}:\Omega \longrightarrow \mathbb{R}^d \times \mathbb{R}^d$ by
\[ 
|| {\boldsymbol V} ||^2_{{\bf L}^2(\Omega)} := 
\int_{\Omega} {\boldsymbol V} : {\boldsymbol V} d \Omega,
\]
where $:$ stands for the matrix (tensor) scalar product.
This allows the introduction of 
\[
{\bf H}^1(\Omega) := \left\{ {\boldsymbol v} \in {\bf L}^2(\Omega):  \|\nabla {\boldsymbol v}\|_{{{\bf L}^2}(\Omega)} < \infty \right\}.
\]
Note that $\nabla {\boldsymbol v}$ is a tensor.
Taking the closure of ${\bf C}^1_0(\Omega)$ (compact support) in ${\bf H}^1(\Omega)$, we get
\[
{\bf H}_0^1(\Omega) := \left\{ {\boldsymbol v} \in {\bf H}^1(\Omega) : {\boldsymbol v}|_{\partial \Omega} = {\bf 0}  \right\}.
\]
Inspired by the singularity removal technique introduced in \citet{Gjerde_2019}, we express $\boldsymbol{u}$ and $\boldsymbol{u}_h$ as 

\begin{equation}
\label{eq:u_splitting}
    {\boldsymbol u} = {\boldsymbol u}^* + {\boldsymbol v} \text{ and }
{\boldsymbol u}_h = {\boldsymbol u}_h^* + {\boldsymbol v}_h,
\end{equation}
where $\boldsymbol{u}^*(\boldsymbol x)$ and $\boldsymbol{u}_h^*(\boldsymbol x)$ are defined in Equations \eqref{eq:u_star} and \eqref{eq:u_h_star}, respectively; $\boldsymbol{v}(\boldsymbol x)$ and $\boldsymbol{v}_h(\boldsymbol x)$ are the auxiliary solutions. Substituting Equation \eqref{eq:u_splitting} into $\mathrm{(BVP_{int})}$ and $\mathrm{(BVP_{sum})}$, together with Equations \eqref{eq:linear_elas_u_star} and \eqref{eq:linear_elas_u_h_star}, $\boldsymbol{v}$ and $\boldsymbol{v}_h$ satisfy the BVPs below, respectively:
\begin{equation*}
    \mathrm{(BVP_{\boldsymbol v})}\left\{
        \begin{aligned}
            -\nabla\cdot\boldsymbol\sigma(\boldsymbol{v}) &= \boldsymbol{0}, &\mbox{$\boldsymbol{x}\in\Omega\subset\mathbb{R}^d$,}\\
            \boldsymbol{v}(\boldsymbol{x}) &= -\boldsymbol{u}^*(\boldsymbol{x}), &\mbox{$\boldsymbol{x}\in\partial\Omega$},
        \end{aligned}
    \right.
\end{equation*}
and
\begin{equation*}
    \mathrm{(BVP_{{\boldsymbol v}_h})}\left\{
        \begin{aligned}
            -\nabla\cdot\boldsymbol\sigma(\boldsymbol{v}_h) &= \boldsymbol{0}, &\mbox{$\boldsymbol{x}\in\Omega\subset\mathbb{R}^d$,}\\
            \boldsymbol{v}_h(\boldsymbol{x}) &= -\boldsymbol{u}_h^*(\boldsymbol{x}), &\mbox{$\boldsymbol{x}\in\partial\Omega$}.
        \end{aligned}
    \right.
\end{equation*}
As neither $\boldsymbol{u}^*$ nor $\boldsymbol{u}^*_h$ lie in ${\bf H}^1(\Omega)$, solutions $\boldsymbol{u}$ and $\boldsymbol{u}_h$ to $\mathrm{(BVP_{int})}$ and $\mathrm{(BVP_{sum})}$, respectively, are not in ${\bf H}^1(\Omega)$ either. However, since $\boldsymbol{G}$ and, therefore $\boldsymbol{u}^*$ and $\boldsymbol{u}_h^*$ are smooth away from $\Gamma$, we have $\boldsymbol{u}^*, \boldsymbol{u}_h^* \in {\bf L}^2(\partial \Omega)$. Note that $\partial \Omega$ is away from $\Gamma$. Of course, away from the point where the Dirac measure acts, the solution is in ${\bf H}^1$.
This is characteristic in general for equations like the Poisson or heat equation, which has been investigated, for instance, in \citep{Peng_2022_JCAM, HMEvers2015, Yang2025, Boon_2023}.  However, the solution lives in ${\bf L}^2(\Omega)$. Bearing in mind that
$${\bf C}^{\infty}_0(\Omega) \subset {\bf H}^1_0(\Omega) \subset {\bf L}^2(\Omega),$$
and since ${\bf C}^{\infty}_0(\Omega)$ is dense in ${\bf L}^2(\Omega)$, it follows that ${\bf H}^1_0(\Omega)$ is dense in ${\bf L}^2(\Omega)$, and therefore the solution can be approximated arbitrarily well by spans of functions in ${\bf H}^1_0(\Omega)$, provided that we take a sufficient number of these functions. However, the convergence rate is jeopardized by the solution not being in ${\bf H}^1_0(\Omega)$. For this reason, the solution is sought
in weighted finite element (Sobolev) spaces \citep{Boon_2023} or in different Sobolev spaces, such as ${\bf W}^{1,1}$ \citep{D_Angelo_2012}, or one assesses the convergence of
the FEM solution away from the point of action of the Dirac delta
distribution \citep{K_ppl_2014}. Boon and Vermolen \cite{Boon_2023} assessed well-posedness in weighted Sobolev spaces in $\mathbb{R}^{2}$ for the equation of elasticity with a Dirac point force. K\"oppl \citep{K_ppl_2024} applied a Zenger correction to assess an elliptic partial differential equation with a Dirac term. 

The convergence of the solutions to $\mathrm{(BVP_{int})}$ and $\mathrm{(BVP_{sum})}$ is verified by analyzing the convergence between $\boldsymbol{v}^*$ and $\boldsymbol{v}_h^*$, as well as $\boldsymbol{u}^*$ and $\boldsymbol{u}_h^*$. For the concision of the manuscript, we denote $\boldsymbol{u}_p = \{\boldsymbol{u},\boldsymbol{u}_h\}$, $\boldsymbol{u}^*_p = \{\boldsymbol{u}^*,\boldsymbol{u}^*_h\}$, and $\boldsymbol{v}_p = \{\boldsymbol{v},\boldsymbol{v}_h\}$, where the asterisk symbolizes the field solution that are obtained by a convolution of the fundamental solution; the subscript $h$ symbolizes the approximation of the solution by the use of a quadrature rule to approximate the integral over $\Gamma$.

We start with the existence and uniqueness of the auxiliary solution $\boldsymbol{v}_p$ $\in {\bf{H}}^1(\Omega)$, then we verify the convergence between $\boldsymbol{v}$ and $\boldsymbol{v}_h$. Consider the weak form of $\mathrm{(BVP_{\boldsymbol v})}$ and $\mathrm{(BVP_{{\boldsymbol v}_h})}$:
\begin{equation*}
    \mathrm{(WF_{{\boldsymbol v}_p})}\left\{
\begin{aligned}
    &\text{Find ${\boldsymbol v}_p \in {\bf H}^1(\Omega)$ subject to ${\boldsymbol v}_p|_{\partial \Omega} = -{\boldsymbol u}_p^*$ such that, for all ${\bf \phi} \in {\bf H}_0^1(\Omega)$,}
%\end{align*}
\\ &\text{ we have }
    a({\boldsymbol v}_p,{\bf \phi)} = 0,
\text{ where } 
a({\boldsymbol v}_{p},~\phi ) = {\int_{\Omega}\boldsymbol{\sigma}{(\boldsymbol {v}_{p})}}:\boldsymbol {\epsilon} (\phi )~d\Omega.
\end{aligned}
    \right.
\end{equation*}
Then the existence and uniqueness of $\boldsymbol{v}_p$ are warranted by Lemma \ref{lemma: uniqueness of v_p} (see Appendix \ref{Appendix} for its proof):
\begin{lemma}\label{lemma: uniqueness of v_p}
    The variational problem $\mathrm{(WF_{{\boldsymbol v}_p})}$ has one and only one solution ${\boldsymbol v}_p \in {\bf H}^1(\Omega)$.
    \label{existence}
\end{lemma}

In cases where the bilinear form satisfies $a(.,.):{\bf H}_1 \times {\bf H}_2 \longrightarrow \mathbb{R}$, where ${\bf H}_1$ and ${\bf H}_2$ are two different Hilbert spaces such as in the case of the use of weighted Sobolev spaces, one uses Necas' Theorem \citep{Brenner_Scott} as a generalization to establish existence and uniqueness. The coercivity condition is replaced with the inf-sup condition and for all ${\boldsymbol u} \in {\bf H}_1$ there must be a ${\boldsymbol v} \in {\bf H}_2$ such that $a(\boldsymbol u,\boldsymbol v) \ne 0$.

As $\boldsymbol{u}_p^*$ is involved as boundary condition for $\mathrm{(BVP_{{\boldsymbol v}_p})}$, the convergence of $\boldsymbol{v}_p$ requires the convergence of $\boldsymbol{u}_p^*$, which solves Equations \eqref{eq:linear_elas_u_star} and \eqref{eq:linear_elas_u_h_star} with integral and summation form of forces, respectively.
We first quantify the difference between the forces, which amounts to a multidimensional generalization of the standard composite midpoint rule, over a polygon (2D) or a closed surface (3D). The numerical validation will be done for this quadrature rule.

\begin{lemma}\label{lemma:midpoint}
    Let $\Gamma_p$ be a polygon, a closed curve consisting of line segments $\Tilde{\Gamma}_q$, embedded in $\Omega \subset \mathbb{R}^2$ that is composed of 
a union of line segments $\Delta \Gamma_j$, that is ${\Gamma_p = \bigcup_{j=1}^m\Delta \Gamma_j}$, such that each $\Tilde{\Gamma}_q$ consists of an integer number of line segments $\Delta \Gamma_p$. Further, let line segment $\Delta \Gamma_j$ be bounded by vertices $\mathbf{x}_j$ and $\mathbf{x}_{j+1}$, such that 
$\mathbf{x}_{m+1} = \mathbf{x}_{1}$, and let $\mathbf{x}_{j+\frac12}$ be the midpoint of a line segment $\Delta \Gamma_j$, and let $f$ be a $C^2$--smooth function that is defined over an open set 
that includes $\Gamma_p$. Then there exists a $C_1>0$, such that
$$
\bigg| \int_{\Gamma_p} f d \Gamma - \sum_{j = 1}^m f(\mathbf{x}_{j+\frac12}) || \mathbf{x}_{j+1} - \mathbf{x}_j ||  \bigg|\le C_1 h^2.
$$
where $h=\text{max}_{j\in{\{1,...,m}\}}\parallel\mathbf{x}_{j+1} - \mathbf{x}_j\parallel$.

In three dimensions, let $\Gamma_p$ be a closed surface in $\Omega \subset \mathbb{R}^3$, that is composed of a union of surface segments $\Delta \Tilde{\Gamma}_q$. Let $\Gamma_p$ be divided by triangular elements $\Delta \Gamma_j$, such that ${\Gamma_p = \bigcup_{j=1}^m\Delta \Gamma_j}$, let $\mathbf{x}_j$ be the centre of $\Delta \Gamma_j$, then there exists a positive constant $C_2$, such that for each component we have
\[
\left| \int_{\Gamma_p} f({\bf x}) \, d{\bf x} - \sum_{j=1}^{m} \frac{|A_j|}{2} f({\bf x}_j) \right| \leq {C_2}h{^2},
\]
where $h$ is the maximal diameter among all the surface elements over $\Gamma_p$ and $| A_j |/2$ is the area of the (triangular) surface element $\Delta \Gamma_j$.
\end{lemma}

\begin{proof}
First we treat the 2D case of a closed curve. Using a midpoint rule for a multivariate function, let us consider
\begin{equation*}
  \mathbf{x}(s)=\mathbf{x}_{j+\frac12}+\frac{s}{2}{(\mathbf{x}_{j+1}-\mathbf{x}_j)}, \quad -1 \leq s \leq 1,
\end{equation*}
Hence, $\mathbf{x}(0)=\mathbf{x}_{j+1/2}\quad\text{and}\quad{\mathbf{x}^{'}(s)}=\frac{1}{2}(\mathbf{x}_{j+1}-\mathbf{x}_j)$, and therefore,
\begin{equation*}
\parallel {\mathbf{x}^{'}(s)}\parallel=\frac{1}{2}\parallel(\mathbf{x}_{j+1}-\mathbf{x}_j)\parallel.
\end{equation*}
We calculate the contribution over a line segment $\Delta \Gamma_j$:
\[
\int\limits_{\Delta\Gamma_j} f(\mathbf{x}) d\Gamma = \int_{-1}^1 f(\mathbf{x}(s))\parallel\mathbf{x^{'}(s)}\parallel ds = \frac{1}{2}\parallel\mathbf{x}_{j+1} - \mathbf{x}_j\parallel \int_{-1}^1 f(\mathbf{x}(s)) ds.
\]
Let ${\boldsymbol{H}(\mathbf x}(s))$ be the Hessian matrix of $f\mathbf{(x}(s))$, then Taylor’s Theorem (see Theorem 1.6.8 in \citep{Vuik2023}) warrants the existence of a $\hat{s}(s)\in (\min(0,s),\max(0,s))$ such that
\begin{align*}
\int\limits_{\Delta\Gamma_j}f(\mathbf{x}) d\Gamma & ~=~ \frac{1}{2}\parallel\mathbf{x}_{j+1} - \mathbf{x}_j\parallel \int_{-1}^1 f(\mathbf{x}(0)) + \frac{s}{2}(\mathbf{x}_{j+1} - \mathbf{x}_j)\nabla f(\mathbf{x}(0)) \\
&\quad + \frac{1}{2}s^2 \left(\frac{\mathbf{x}_{j+1} - \mathbf{x}_j}{2} \right)^T {\boldsymbol{H}(\mathbf{x}}(\hat{s}(s)))(\frac{\mathbf{x}_{j+1} - \mathbf{x}_j}{2})ds \\
& =\frac{1}{2}\parallel\mathbf{x}_{j+1} | - \mathbf{x}_j\parallel \left[2 ~ f(\mathbf{x}_{j+1/2})+0 + \frac{1}{8} \int_{-1}^{1} s^2 (\mathbf{x}_{j+1} - \mathbf{x}_j)^T{\boldsymbol{H}(\mathbf{x}}(\hat{s}(s)))(\mathbf{x}_{j+1} - \mathbf{x}_j)ds\right]\\
&=\parallel\mathbf{x}_{j+1} - \mathbf{x}_j\parallel f(\mathbf{x}_{j+1/2})+ \frac{1}{16}\parallel\mathbf{x}_{j+1} - \mathbf{x}_j  \parallel \int_{-1}^1 s^2(\mathbf{x}_{j+1} - \mathbf{x}_j)^T{\boldsymbol{H}(\mathbf{x}}(\hat{s}(s)))(\mathbf{x}_{j+1} - \mathbf{x}_j)ds.\\
\end{align*}
Since $f$ is smooth in an open region around $\Gamma_p$, it follows that there exists a $\hat{K}_j>0$ (maximum eigenvalue), such that
\begin{equation*}
\big|{(\mathbf x,{\boldsymbol H}(\mathbf x))}\big|\leq \hat{K}_j\parallel\mathbf{x}\parallel^2,\\ \text{on } \Delta \Gamma_j.
\end{equation*}
Therefore, we obtain that
\begin{align*}
     \bigg|\int\limits_{\Delta\Gamma_j}f(\mathbf{x}) d\Gamma-\parallel\mathbf{x}_{j+1} - \mathbf{x}_j\parallel f(\mathbf{x}_{j+1/2})\bigg|\leq \frac{1}{24} \hat{K}_j \parallel\mathbf{x}_{j+1} - \mathbf{x}_j\parallel^3.
\end{align*}

Let $\hat{K} = \max_j (\hat{K}_j)$, then summation of the boundary elements over ${\Gamma}_p$ gives
\begin{align*}
&\bigg|\int\limits_{\Gamma_p}{f(\mathbf{x})} d\Gamma- \sum_{j=1}^{m} \parallel\mathbf{x}_{j+1} - \mathbf{x}_j\parallel f(\mathbf{x}_{j+\frac12})\bigg| \le 
   \sum_{j=1}^{m}\bigg|\int\limits_{\Delta\Gamma_j}{f(\mathbf{x})} d\Gamma-\parallel\mathbf{x}_{j+1} - \mathbf{x}_j\parallel f(\mathbf{x}_{j+\frac12})\bigg|  \\
  \leq&\frac{1}{24} \hat{K} \sum_{j=1}^{m}\parallel\mathbf{x}_{j+1} - \mathbf{x}_j\parallel^3
   \leq \frac{1}{24}\hat{K}\ \text{max}_{j\in{\{1,...,m}\}} (\parallel\mathbf{x}_{j+1} - \mathbf{x}_j\parallel^2)\cdot \big|\Gamma_p\big| 
   =\frac{1}{24}\hat{K}h^2\big|\Gamma_p\big|,
\end{align*}
where $h=\text{max}_{j\in{\{1,...,m}\}}\parallel\mathbf{x}_{j+1} - \mathbf{x}_j\parallel$, and $|{\Gamma }_p|$ is the perimeter of the polygon ${\Gamma }_p$. 
\\[2ex]
Subsequently, we treat the three dimensional case. In three dimensions, the surface element over a manifold is a patch of a surface (a portion of a plane, i.e a triangle). Suppose $e_j$ is a triangular surface element in three dimensional space with vertices $({\bf{x}}_1, {\bf{x}}_2, {\bf{x}}_3)$. We compute the midpoint ${\bf{x}}_c$ of $e_j$ as ${\bf{x}}_c= \frac{1}{3} ({\bf{x}}_1 + {\bf{x}}_2 + {\bf{x}}_3)$.  We map the triangle in $(x, y,z)-$space to the reference triangle in $(s, t)$-space with points $(s_1,t_1)=(0, 0), ~(s_2,t_2)=(1, 0)$ and $(s_3,t_3)=(0, 1)$. The parameterization from the reference triangle $e_0$ to the original (physical) triangle $e_j$ is given by
\begin{equation*}
    {\bf{x}}(s, t) = {\bf{x}}_1(1 - s - t) + s{\bf{x}}_2 + t {\bf{x}}_3, \quad 0 \leq s \leq 1, \quad 0 \leq t \leq 1-s.
\end{equation*}
For any function $f({\mathbf{x}}) \in {\bf{C}^2(\Omega)}$, the integral of $f(\mathbf{x})$ over the original triangle is given by
\begin{equation*}
    {\int_{e_j}} f({\bf{x}})d{{\bf{x}}} = {\int_{e_0}} f({\bf{x}}(s, t)) |A_j| d(s, t),
\end{equation*}
where $|A_j|=\sqrt{| \det({\bf{J}}^T {\bf{J}})|}$, where $\bf{J}$ is the Jacobian matrix, and $e_0 = \{(s, t) \in \mathbb{R}^2 : 0 \leq s \leq 1, 0 \leq t \leq 1 - s\}$ given by
\begin{equation*}
   {\bf{J}} = \frac{\partial(x, y,z)}{\partial(s, t)} = \begin{pmatrix} x_2 - x_1 & x_3 - x_1 \\ y_2 - y_1 & y_3 - y_1 \\ z_2 - z_1 & z_3 - z_1 \end{pmatrix},
\end{equation*}
and $\sqrt{| \det({\bf{J}}^T {\bf{J}})|}$ is twice the area of the original triangle $e_j$, i.e.,
\begin{equation*}
    |A_j| := \sqrt{| \det({\bf{J}}^T {\bf{J}})|} = \|({\bf{x}}_2 - {\bf{x}}_1) \times ({\bf{x}}_3 - {\bf{x}}_1)\|.
\end{equation*}
Let ${\bf{x}}(1/3, 1/3)$ coincide with the midpoint ${\bf{x}}_c$ of element $e_j$, using the midpoint rule (Taylor's Theorem) for a multivariate function gives

\begin{align*}
    \int_{e_j} f(\mathbf{x})d{\mathbf {x}} &~  
    = \int_{e_j} \left(f({\bf{x}}_c) + ({\bf{x}} - {\bf{x}}_c) \cdot \nabla {f}({\bf{x}}_c) + \frac{1}{2} ({\bf{x}} - {\bf{x}}_c)^T {\boldsymbol{H}}({\bf{x}}_c + \theta({\bf x}) ({\bf{x}}_c - {\bf{x}})) ({\bf{x}} - \mathbf {x}_c) \right) d{\mathbf {x}} \\
 &=  \frac{1}{2} |A_j| f({\bf{x}}_c) + 0 + \frac{1}{2} \int_{e_j}({\bf{x}} - {\bf{x}}_c)^T {\boldsymbol{H}}({\bf{x}}_c + \theta({\bf x}) ({\bf{x}}_c - {\bf{x}})) ({\bf{x}} - \mathbf {x}_c) d{\mathbf {x}}, ~ \exists ~ \theta({\bf x}) \in (0,1). 
\end{align*}
Since $f(\mathbf{x}) \in {\bf C}^2(\Omega)$, it follows that there exists a $\hat{K}_j>0$ (maximum eigenvalue, different from the 2D-case), such that
\begin{equation*}
\big|{({\mathbf x},{\boldsymbol H}({\mathbf x}))}\big|\leq \hat{K}_j\parallel\mathbf{x}\parallel^2, \text{on } e_j.
\end{equation*}
Therefore, we obtain that
\[
\left| \int_{e_j} f{\bf{(x)}} \, d{\bf{x}} - \frac{ |A_j|}{2} f{\bf{(x}}_c) \right| \leq \left| \frac{ 1}{2} \int_{e_j} ({\bf{x}} - {\bf{x}}_c)^T {\boldsymbol{H}}({\bf{x}}_c + \theta({\bf x}) ({\bf{x}}_c - {\bf{x}})) ({\bf{x}} - \mathbf {x}_c) \,d{\mathbf {x}} \right| \leq \frac{ |A_j|}{4}\hat{K}_j{h^2},
\]
where $h$ is the maximal diameter in the original triangle $e_j$. \\
Let $\hat{K} = \max_j (\hat{K}_j)$, considering all the surface elements over $\Gamma_p$, we get
\[
\left| \int_{\Gamma_p} f{\bf{(x}}) \, d{\bf{x}} - \sum_{j=1}^{m} \frac{|A_j|}{2} f{\bf{(x}}_j) \right| \leq \frac{\hat{K} {h^2}}{4} \sum_{j=1}^{m} \frac{ |A_j|}{2} = \frac{\hat{K} {h^2}}{4} |\Gamma_p|,
\]
where $h$ is the maximal diameter among all the surface elements (i.e., triangles) and $|\Gamma_p|$ is the sum of the areas (area in $\mathbb{R}^3$) of all the surface elements over $\Gamma_p$. Therefore, in three dimensions, we can conclude that there exists a constant $C_2 > 0$ such that:
\[
\left| \int_{\Gamma_p} f{\bf{(x}}) \, d{\bf{x}} - \sum_{j=1}^{m} \frac{|A_j|}{2} f{\bf{(x}}_j) \right| \leq {C_2}{h^2}.
\] 
\end{proof}
In principle, we demonstrated the convergence of the midpoint rule for quadrature in a more generic setting. We note that this result for a generic curve or manifold was also demonstrated in other works \citep{Peng_2022}, among others, although it was not labeled as a formal result. Different quadrature strategies may be used, such as Newton-Cotes or Gaussian-based quadratures, which have their respective orders of convergence. For a generic convergence result for the computation of $\int_{\Gamma_p} f(\mathbf{x}) d \mathbf{x}$ with 'sufficient' smoothness, one obtains an error that is bounded by $C_2 h^r$ for a certain $C_2 > 0,~ r > 0$, where $r$ represents the order of convergence. Assuming such an error bound, we arrive at the following important results. 
\begin{proposition}
\label{corr1}
    Let $\boldsymbol{u}^{\ast }$ and $\boldsymbol{u}_{h}^{\ast }$, respectively, be given by equation (\ref{eq:u_star}) and (\ref{eq:u_h_star}), and suppose that the numerical quadrature method over the closed surface or contour over $\Gamma_p$ is bounded from above by $C h^r$ for a $C > 0$, $r > 0$ then for all $\boldsymbol{x} \in \Omega_\delta=\{\boldsymbol{x} \in \Omega \ {|} \ {\rm dist} (\boldsymbol{x},\Gamma) >\delta\}$, $\delta > 0$ (that is away from $\Gamma$), there is a $C_2 > 0$ and a $\Bar{C} > 0$ such that for $\boldsymbol{x} \in \Omega_\delta$, 
$|{\boldsymbol u}_h^* - {\boldsymbol u}^*| \le C_2 h^{r}$, and   $\parallel \boldsymbol{u}_{h}^*-\boldsymbol{u}^*\parallel_{{\bf L}^2(\Omega_{\delta})} \leq \Bar{C}h^{r}$. Note that here $\delta$ represents a positive constant but does not represent the Dirac delta distribution.
\end{proposition}
\begin{proof}
    Since ${\bf G}(\boldsymbol{x};\boldsymbol{x}')$ is smooth for any $\boldsymbol{x} \in \Omega_\delta$ (away from $\Gamma$), and hence the requirements of Lemma \ref{lemma:midpoint} are fulfilled, we have $|{\boldsymbol u}_h^* - {\boldsymbol u}^*| \le C_2 h^{r}$ for any $\boldsymbol{x} \in \Omega_\delta$. Integration over $\Omega_\delta$ gives
\begin{equation*}
\begin{split}  
 \left[\int_{\Omega_\delta}| \boldsymbol{u}_{h}^*-\boldsymbol{u}^*|^{2} d{\Omega}\right]^{1/2}\leq \left[\int_{\Omega_\delta}C_2^{2}({\boldsymbol{x}})h^{2r} d{\Omega}\right]^{1/2} \le \sup_{\boldsymbol{x} \in \Omega_\delta} C_2({\boldsymbol{x}})h^{r} |{\Omega_\delta}|^{1/2}= \Bar{C}h^{r}.
\end{split}
\end{equation*}
Hence,
\begin{equation}
\parallel \boldsymbol{u}_{h}^*-\boldsymbol{u}^*\parallel _{{\bf L}^{2}(\Omega_\delta
)}\leq \Bar{C}h^{r}.
\label{one}
\end{equation}
\end{proof}

\begin{proposition}
\label{corr2}
    Let $\boldsymbol{u}^{\ast }$ and $\boldsymbol{u}_{h}^{\ast }$, respectively, be given by equation (\ref{eq:u_star}) and (\ref{eq:u_h_star}), and suppose that the numerical quadrature method over the closed surface or contour over $\Gamma_p$ is bounded from above by $C h^r$ for a $C > 0$, $r > 0$. Further let $S$ and $\Tilde{\Gamma}$ be a bounded surface and curve, then there are $C_S > 0$ and $C_{\tilde{\Gamma}}>0$, such that $\parallel \boldsymbol{u}_{h}^*-\boldsymbol{u}^*\parallel_{{\bf L}^2(S)} \leq {C_S}h^{r}$ and $\parallel \boldsymbol{u}_{h}^*-\boldsymbol{u}^*\parallel_{{\bf L}^2(\tilde{\Gamma})} \leq {C_{\tilde{\Gamma}}}h^{r}$, respectively.
\end{proposition} 
 \begin{proof}
     Since ${\bf G}(\boldsymbol{x};\boldsymbol{x}')$ is smooth for any $\boldsymbol{x} \in S$, and hence the requirements of Lemma \ref{lemma:midpoint} are fulfilled, we have $|{\boldsymbol u}_h^* - {\boldsymbol u}^*| \le C_2 h^{r}$ for any $\boldsymbol{x} \in S$. Integration over $S$ gives

\begin{equation*}
\begin{split}  
 \left[\int_{S}| \boldsymbol{u}_{h}^*-\boldsymbol{u}^*|^{2} d{S}\right]^{1/2}\leq \left[\int_{S}C_2^{2}({\boldsymbol{x}})h^{2r} d{S}\right]^{1/2}\le \sup_{\boldsymbol{x} \in S}C_2({\boldsymbol{x}})h^{r} |{S}|^{1/2}=C_S h^{r}.
\end{split}
\end{equation*}
Hence,
\begin{equation}
\parallel \boldsymbol{u}_{h}^*-\boldsymbol{u}^*\parallel _{{\bf L}^{2}(S
)}\leq {C_S}h^{r}.
\label{two}
\end{equation}
and analogously,
\begin{equation}
\parallel \boldsymbol{u}_{h}^*-\boldsymbol{u}^*\parallel _{{\bf L}^{2}(\Tilde{\Gamma}
)}\leq {C_{\tilde{\Gamma}}}h^{r}.
\label{three}
\end{equation}
\end{proof}

The above lemmas and propositions lead to one of the key results in this manuscript, regarding the convergence between the auxiliary solutions $\boldsymbol{v}$ and $\boldsymbol{v}_h$:
\begin{theorem}
\label{Th_v_v_h}
    Let $\boldsymbol{v}$ and $\boldsymbol{v}_h$ solve $\mathrm{(BVP_{\boldsymbol v})}$ and $\mathrm{(BVP_{{\boldsymbol v}_h})}$, respectively, then there is a $K > 0$ such that $|| \boldsymbol{v} - \boldsymbol{v}_h ||_{{\bf L}^2(\Omega_\delta)} \le K h^{r}$.
\end{theorem}
\begin{proof}
From Lemma \ref{lemma: uniqueness of v_p}, we have existence of $\boldsymbol{v},~\boldsymbol{v}_h \in {\bf H}^1(\Omega)$, as well as of their difference. With Proposition \ref{corr2}, one obtains
\begin{align*}
||{\boldsymbol u}_h^* - {\boldsymbol u}^* ||^2_{{\bf L}^2(\partial \Omega)}\le C_3h^{r}, \quad \exists ~ C_3 > 0.
\end{align*}
Note that the right-hand side is a constant, hence in ${\bf H}^{1/2}(\partial \Omega) \subset {\bf L}^2(\partial \Omega)$ (recall from Equation (\ref{eq:u_splitting}) that $\{\boldsymbol{v},~\boldsymbol{v}_h\}=-\{\boldsymbol{u}^*,~\boldsymbol{u}_h^*\}$ on $\partial \Omega$), therefore
since
\begin{align*}
||{\boldsymbol v}_h - {\boldsymbol v} ||_{{\bf L}^2(\partial\Omega)}= 
|| {\boldsymbol u}_h^* - {\boldsymbol u}^* ||_{{\bf L}^2(\partial\Omega)} \le C_3h^{r}.
\end{align*}
The Trace Extension Theorem \citep{adams2003sobolev}
asserts the existence of a $\beta > 0$ such that
\begin{align*}
|| {\boldsymbol v}_h - {\boldsymbol v} ||_{{\bf H}^1(\Omega)} \le \beta
|| {\boldsymbol v}_h - {\boldsymbol v} ||_{{\bf L}^2(\partial \Omega)} = \beta
|| {\boldsymbol u}_h^* - {\boldsymbol u}^* ||_{{\bf L}^2(\partial \Omega)} \le \beta C_3h^{r}.
\end{align*}
Since
\begin{align*}
|| {\boldsymbol v}_h - {\boldsymbol v} ||_{{\bf L}^2(\Omega)} \le || {\boldsymbol v}_h - {\boldsymbol v} ||_{{\bf H}^1(\Omega)} \le C_4 h^{r}, \qquad C_4 = \beta ~ C_3,
\end{align*}
it follows that
there is a $K>0$ such that
\begin{align}
|| {\boldsymbol v}_h - {\boldsymbol v} ||_{{\bf L}^2(\Omega_\delta)} \le || {\boldsymbol v}_h - {\boldsymbol v} ||_{{\bf L}^2(\Omega)} \le K h^{r},
\label{two}
\end{align}
 \end{proof}

Subsequently, after having proved the convergence between $\boldsymbol{v}$ and $\boldsymbol{v}_h$, as well as between $\boldsymbol{u}^*$ and $\boldsymbol{u}_h^*$, the convergence between the solutions to $\mathrm{(BVP_{int})}$ and $\mathrm{(BVP_{sum})}$ (namely, $\boldsymbol{u}$ and $\boldsymbol{u}_h$) is warranted by the theorem below. 
\begin{theorem}
\label{Th_u_u_h}
    Let $\boldsymbol{u}$ and $\boldsymbol{u}_h$ be the solutions to $\mathrm{(BVP_{int})}$ and $\mathrm{(BVP_{sum})}$, respectively. Then for all $\boldsymbol{x} \in \Omega_\delta=\{\boldsymbol{x} \in \Omega \ | \ dist (\boldsymbol{x},\Gamma) >\delta\}$, $\delta > 0$ (that is away from $\Gamma$), there is a $C_0 > 0$ such that
$\parallel {\boldsymbol u}_h - {\boldsymbol u}\parallel_{{\bf L}^2(\Omega_{\delta})} \le C_0 h^{r}$.
\end{theorem}
\begin{proof}
    Since $||.||_{{\bf L}^2(\Omega_\delta)}$ is a proper norm,  we apply the Triangular Inequality and combine the results from Proposition \ref{corr1} and Relation (\ref{two}), in order to obtain
\begin{align*}
    \parallel {\boldsymbol{u}}_{h}-{\boldsymbol{u}} \parallel_{{\bf L}^2(\Omega_\delta)}
 &= \parallel \boldsymbol{u}_{h}^{*} +\boldsymbol{v}_{h} -(\boldsymbol{u}^{*}
+\boldsymbol{v})\parallel_{{\bf L}^2(\Omega_\delta)}= \parallel \boldsymbol{u}_{h}^{*}-\boldsymbol{u}^{*}
+\boldsymbol{v}_{h}-\boldsymbol{v}\parallel_{{\bf L}^2(\Omega_\delta)}\\
 &\leq \parallel \boldsymbol{u}_{h}^{*}-\boldsymbol{u}^{*}\parallel_{{\bf L}^2(\Omega_\delta)}
+\parallel \boldsymbol{v}_{h}-\boldsymbol{v}\parallel_{{\bf L}^2(\Omega_\delta)}
\le (C+K)h^{r}=C_0 h^{r}.
\end{align*}
\end{proof}

Note that the above result is valid for two- and three-dimensional cases, and that this amounts to convergence being determined by the quadrature rule that is used over the contour or surface where the forces are applied. This result agrees with common intuition.

\section{Numerical Validation}
\label{sec:4}
In this section, we conduct the simulations in two and three dimensions to numerically verify the convergence between $\boldsymbol{u}^*$ and $\boldsymbol{u}_h^*$ and between $\boldsymbol{u}$ and $\boldsymbol{u}_h$ in a predefined subset in $\Omega_\delta$, which excludes the positions where point forces are applied. In particular, we evaluate solutions (i.e. $\boldsymbol{u}^*_h, \boldsymbol{u}_h$ and $\boldsymbol{v}_h$) over a curve (or a surface in three dimensions), which is denoted by $S\subset \Omega_\delta$ in line with Figure \ref{fig1}.

The parameter values are listed in Table \ref{Table:para_2D} for both two- and three-dimensional simulations. The computational domain is $\Omega = (-2,2)^d$ with $d=\{2, 3\}$, the cell and the closed curve $S$ is centered at the origin with radius $r$ and $R$, respectively. For a two-dimensional case, the FEM was implemented in Python 3.12 using the FEniCS package version 2019.2.0.dev0 \citep{FEniCS}.

\begin{table}[h!]\footnotesize
\centering
\caption{\it Parameter values for two- and three dimensional simulations in Section \ref{sec:4}.}
\label{Table:para_2D}
\begin{tabular}{lllll}
\hline
Parameter & Description & Value & Units & Source \\ \hline
$E$ & Young's Modulus & $1.0\times 10^{7}$ & $kg/(\mu m \cdot\text{ }min^{2})$ & \citet{Peng_Cancer_2023} \\ 
$Q$ & Magnitude of the force exerted on the cell & $1.0\times 10^{3}$ & $%
kg \cdot\mu m/min^{2}$ &  \citet{Peng_Cancer_2023} \\ 
$R$ & Radius of the boundary outside the cell & $0.5$ & $\mu m$ & Estimated in this study \\ 
$r$ & Radius of the cell & $0.3$ & $\mu m$ & \citet{Chen_2017} \\ 
$\nu $ & Poisson's Ratio & $0.25$ & $-$ &  \citet{Peng_Cancer_2023}\\ \hline
\hline
\end{tabular}
\end{table}

\subsection{Two Dimensional Simulations}
We model a regular polygonal curve $\Gamma$ centered at the origin in $\mathbb{R}^2$ that exerts radial pulling forces at the midpoints of its edges. As the number of polygon edges (line segments) approaches infinity, the polygon tends to a circle. %The magnitudes of the parameter values can be found in Table \ref{Table:para_2D}. 
Figure \ref{fig:2-8} displays the resulting displacement $\boldsymbol{u}_h$ from $\mathrm{(BVP_{sum})}$ field within the domain $[-2,2] \times [-2,2]$, when the cell is pulling the immediate environment. Figure \ref{fig:2-8}(a) shows the magnitude of the displacement and Figure \ref{fig:2-8}(b) zooms in the certain region and shows the vector field of $\boldsymbol{u}_h$. 

\begin{figure}[h!]
%\ center
\centering
\subfigure[Displacement magnitude]{
\includegraphics[width=0.48\textwidth]{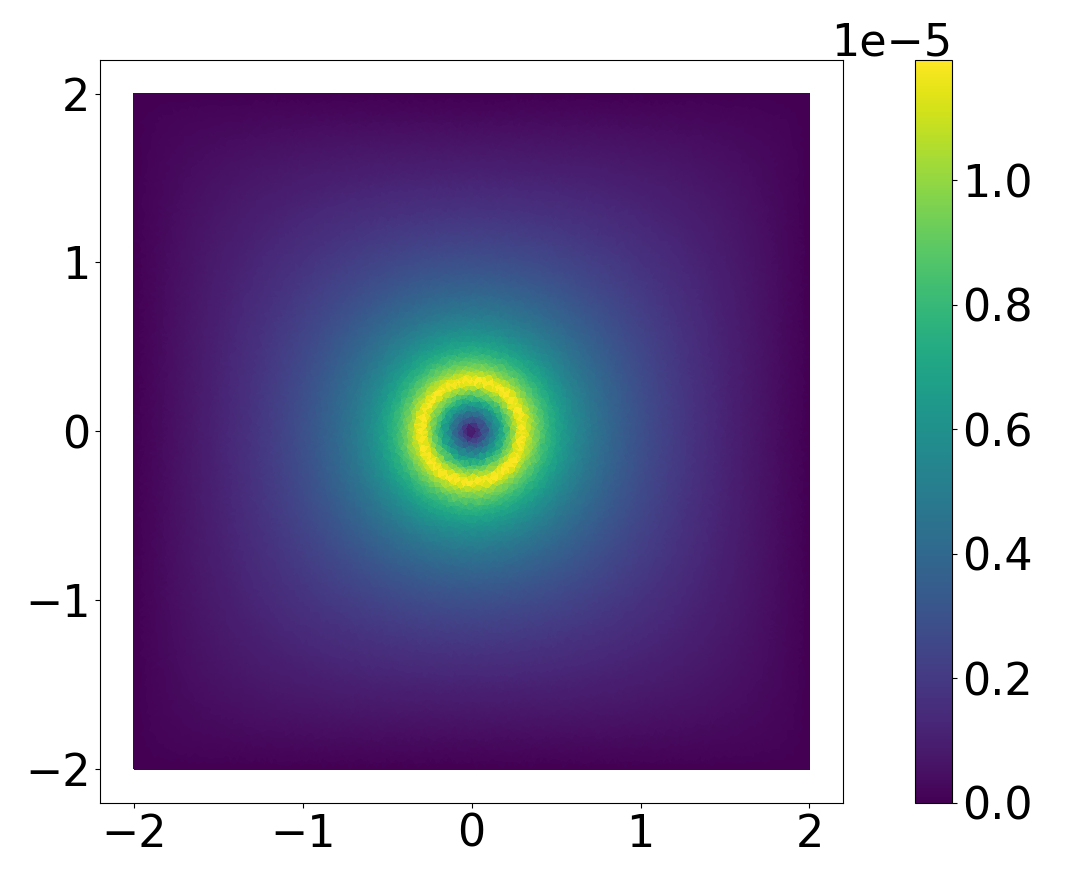}}
\subfigure[Displacement field]{\includegraphics[width=0.48\textwidth]{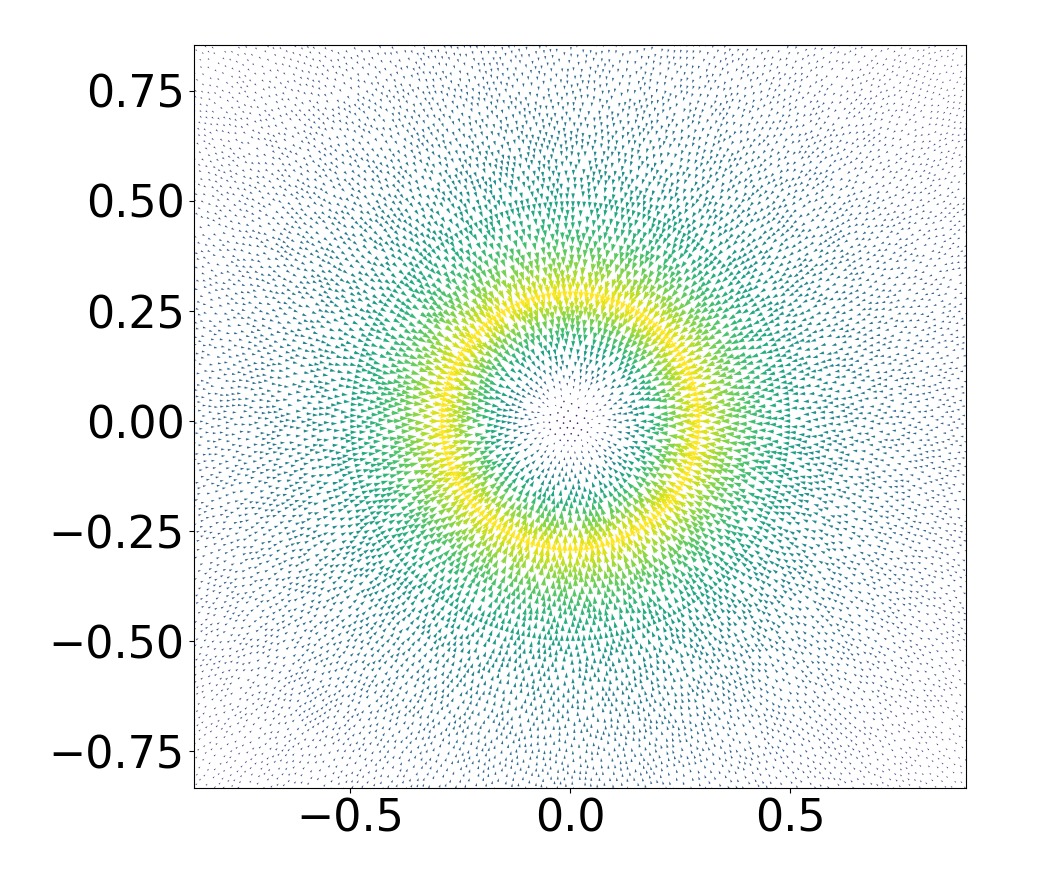}}
    \caption{\it Magnitude of the displacements and displacement field around a circular cell of radius $r=0.3\, \mu m$, where the point forces are applied, are shown in the sub-figurrs (a) and (b), respectively. For the displacement field (sub-figure (b)), it is zoomed in. The input parameters can be found in Table \ref{Table:para_2D}.}
    \label{fig:2-8}
\end{figure}
We evaluate the convergence of the solution ${\boldsymbol u}_h^*$ with respect to the number of mesh points, $m$, on the cell boundary. Convergence is quantified by the ${\bf L}^2$--norm over the circle centered at the origin and radius $R = 0.5$. To approximate the ${\bf L}^2$--norm over this circle $\partial B(\boldsymbol{0},R)$, the circle is divided into $N$ segments so that a polygon results. The length of each line segment is $\Delta \Gamma$, and the midpoint of the line segment of the polygon is denoted by $\boldsymbol{x}_j$. Then the ${\bf L}^2$--norm over $\partial B(\boldsymbol{0}, R)$ can be numerically computed by 
\begin{equation}
|| {\boldsymbol u}_h^* ||_{{\bf L}^2(\partial B(0,R))}^2 = \int_{\partial B(0,R)} ||{\boldsymbol u}_h^*||^2 d \Gamma  
\approx \sum_{j = 1}^N || {\boldsymbol u}_h^*(\boldsymbol{x}_j)||^2 \Delta \Gamma.
\end{equation}
The finite element discretization has been chosen such that $h \sim \Delta \Gamma$.
The numerical error of this composite midpoint rule is given by $\mathcal{O}(\Delta \Gamma^2)$.
Richardson's estimate of the order of convergence \cite{van2005numerical} is used to estimate the order of convergence. 

Table \ref{tab:boundary_edge_results}-\ref{tab:2D_convergence_fixed_FEM_mesh} present the order of convergence of ${\bf L}^2$--norm of $\boldsymbol{u}_h$ (solution to $\mathrm{(BVP_{sum})}$ with singularity removal technique), $\boldsymbol{u}_h^{\rm FEM}$ (solution to $\mathrm{(BVP_{sum})}$ using classical finite-element method), $\boldsymbol{u}^*_h$ (fundamental solution expressed in Equation \eqref{eq:u_h_star}) and $\boldsymbol{v}_h$ (solution to $\mathrm{(BVP_{{\boldsymbol v}_h})})$. In Table \ref{tab:boundary_edge_results}, the cell boundary is divided into line segments by FEM mesh (i.e., the endpoints of line segments over the cell boundary coincide with the nodal points of FEM mesh), that is, the number of edges of cell boundary doubles due to the refinement of FEM mesh. Hence, the numerical errors result from both FEM and the quadrature rule in Equation \eqref{eq:linear_elas_u_h_star}. Meanwhile, Table \ref{tab:2D_convergence_fixed_FEM_mesh} utilizes a high-resolution FEM mesh with average mesh size of $0.01912$. Therefore, the results in this table only reflect the error from the quadrature rule over the interface (cell boundary) $\Gamma$.

For both aforementioned simulation settings, the order of convergence of the ${\bf L}^2$--norm is about two, which indicates excellent agreement with Proposition \ref{corr2} and Theorems \ref{Th_v_v_h} - \ref{Th_u_u_h}. Note that the convergence of $\boldsymbol{u}_h^{\mathrm{FEM}}$ is less satisfying compared to $\boldsymbol{u}_h$ solved by the singularity removal technique. This convergence characteristic is in line with the findings by \cite{K_ppl_2014} of convergence of finite element solutions for Laplace equations with Dirac measures, where it was found that the ${\bf L}^2$--error away from the singularity behaves like $\mathcal{O}(h^2 \ln(h))$ for linear Lagrangian elements.

\begin{table}[h!]\footnotesize
\centering
\caption{\it Weighted ${\bf L}^2$--norm and the corresponding order of convergence for circular domain in the two-dimensional case for the closed curve $S$, as indicated in Figure \ref{fig1}. The number of edges on the cell boundary is a consequence of FEM mesh refinement.}
\label{tab:boundary_edge_results}
\begin{tabular}{m{1.5cm}m{1cm}|m{2.5cm}m{1.5cm}||m{2.5cm}m{1.5cm}}
\hline
\textbf{Number of edges on cell boundary} & \textbf{FEM mesh size} & \textbf{$\left\Vert \mathbf{u_h}%
\right\Vert _{{\bf L}^{2}(S)}$} & \textbf{Order of convergence} & \textbf{$\left\Vert \mathbf{u}%
_h^{\mathrm{FEM}}\right\Vert _{{\bf L}^{2}(S)}$} & \textbf{Order of convergence} \\
\hline
9   & $0.2416$ &$1.0259604 \times 10^{-5}$ & --             & $1.0359142 \times 10^{-5}$ & --             \\
18  & $0.1208$ & $1.0603018 \times 10^{-5}$ & --             & $1.0631187 \times 10^{-5}$ & --             \\
36  & $0.0604$ &$1.0694758 \times 10^{-5}$ & 1.904331227    & $1.0700364 \times 10^{-5}$ & 1.975480139    \\
72  & $0.0302$   & $1.0718094 \times 10^{-5}$ & 1.9749689      & $1.0719390 \times 10^{-5}$ & 1.862294754    \\
\hline
\hline
\end{tabular}

\begin{tabular}{m{1.5cm}m{1cm}|m{2.5cm}m{1.5cm}||m{2.5cm}m{1.5cm}}
\textbf{Number of edges on cell boundary} & \textbf{FEM mesh size} & \textbf{$\left\Vert \mathbf{u_h}{^*}\right\Vert _{{\bf L}^{2}(S)}$} & \textbf{Order of convergence} & \textbf{$\left\Vert \mathbf{v_h}%
\right\Vert _{{\bf L}^{2}(S)}$} & \textbf{Order of convergence} \\
\hline
9  & $0.2416$ & $1.0801511 \times 10^{-5}$ & --              & $5.4240216 \times 10^{-7}$ & --              \\
18 & $0.1208$ & $1.1148931 \times 10^{-5}$ & --              & $5.4614627 \times 10^{-7}$ & --              \\
36 & $0.0604$ & $1.1241513 \times 10^{-5}$ & 1.907866181     & $5.4708112 \times 10^{-7}$ & 2.001824342     \\
72 & $0.0302$ & $1.1265049 \times 10^{-5}$ & 1.975908187     & $5.4731631 \times 10^{-7}$ & 1.990877149     \\
\hline
\hline
\end{tabular}
\end{table}

\begin{table}[h!]
\footnotesize
    \centering
    \caption{\it Weighted ${\bf L}^2$--norm and the corresponding order of convergence for circular domain in the two-dimensional case for the closed curve $S$, as indicated in Figure \ref{fig1}. Note that the FEM mesh is fixed regardless of the increasing number of edges on the cell boundary. Here, the FEM mesh size is $0.01912$.}
    \begin{tabular}{{m{1.5cm}|m{2.5cm}m{1.5cm}||m{2.5cm}m{1.5cm}}}
    \hline
    \textbf{Number of edges on cell boundary} & \textbf{$\left\Vert \mathbf{u_h}%
    \right\Vert _{{\bf L}^{2}(S)}$} & \textbf{Order of convergence} & \textbf{$\left\Vert \mathbf{u}%
    _h^{\mathrm{FEM}}\right\Vert _{{\bf L}^{2}(S)}$} & \textbf{Order of convergence} \\
    \hline
    10 &$1.1606485 \times 10^{-5}$ & --             & $1.1605940 \times 10^{-5}$ & --             \\
    20 & $1.2375921 \times 10^{-5}$ & --             & $1.2375388 \times 10^{-5}$ & --             \\
    40 &$1.2549946 \times 10^{-5}$ & 2.1444506522    & $1.2549403\times 10^{-5}$ & 2.144607035    \\
    80 & $1.2591117 \times 10^{-5}$ & 2.079605617      & $1.2590573 \times 10^{-5}$ & 2.079559888    \\
    \hline
    \hline
    \end{tabular}

    \begin{tabular}{m{1.5cm}|m{2.5cm}m{1.5cm}||m{2.5cm}m{1.5cm}}
    \textbf{Number of edges on cell boundary}  & \textbf{$\left\Vert \mathbf{u_h}{^*}\right\Vert _{{\bf L}^{2}(S)}$} & \textbf{Order of convergence} & \textbf{$\left\Vert \mathbf{v_h}%
    \right\Vert _{{\bf L}^{2}(S)}$} & \textbf{Order of convergence} \\
    \hline
    10  & $1.2241133 \times 10^{-5}$ & --              & $6.3486279 \times 10^{-7}$ & --              \\
    20 & $1.3052789 \times 10^{-5}$ & --              & $6.7702553 \times 10^{-7}$ & --              \\
    40 & $1.3236332 \times 10^{-5}$ & 2.144751170     & $6.8654559 \times 10^{-7}$ & 2.146926574     \\
    80 & $1.3279755 \times 10^{-5}$ & 2.079605661     & $6.8879784 \times 10^{-7}$ & 2.079604572     \\
    \hline
    \hline
    \end{tabular}
    
    \label{tab:2D_convergence_fixed_FEM_mesh}
\end{table}

\subsection{Three Dimensional Simulations}
In the three dimensional set-up, we consider a regular surface $\Gamma$ composed of triangular faces and radially pulling forces are exerted from the barycenter of the faces inwards. As the number of faces increases, the surface $\Gamma$ approaches the sphere $\partial B({\bf 0},r)$. As in the two dimensional case, we evaluate the convergence of the fundamental solutions ${\boldsymbol u}_h^{\ast}$, Eq. \eqref{eq:u_h_star}, with respect to the number of faces on the boundary of the cell. Convergence is quantified by the ${\bf L}^2$--norm over the parallelograms as sketched in Figure \ref{fig2}, where we consider two cases: 
\begin{itemize}
    \item[-] 
a plane not intersecting the cell $\Gamma$, that is, $S \cap \Gamma = \emptyset$;
\item[-] a plane intersecting the cell $\Gamma$, that is, $S \cap \Gamma \ne \emptyset$. 
\end{itemize}
We consider successive refinements of $\Gamma$, whereas the mesh for the quadrature on $S$ is not further refined. For each refinement, each triangular face in $\Gamma$ is divided into four similar triangles.
Furthermore, the region $S$ is divided into 2500 uniform sub-parallelograms, and the Midpoint Rule for numerical integration is used.

Using the input data from Table \ref{Table:para_2D}, and taking the region enclosed by $\Gamma$ centered around the origin, we consider the situations in Figure \ref{fig2}. The results for the ${\bf L}^2$--norm have been listed in Table \ref{Table:results_L2_norm_3D_C1}, and in both cases a quadratic $\mathcal{O}(h^2)$ convergence is observed. The observed convergence rate agrees excellently with the predicted theoretical rate. It is a key result of this study which strongly supports the validity of our proposed methodology.
Note that $h$ represents a measure for the largest diameter of the triangular elements on $\Gamma$.
In the first case (Figure \ref{fig2}(a)), $S$ represents a parallelogram that does not intersect the spherical surface $\Gamma$, and the result is entirely in line with our findings. In the second case (Figure \ref{fig2}(b)), $S$ represents a parallelogram that intersects the sphere $\Gamma$. The observed quadratic convergence is attributed to the fact that the fundamental solution is in ${\bf L}^2(\mathbb{R}^3) \setminus {\bf H}^1(\mathbb{R}^3)$, and to the fact that the intersection, being a curve, has zero measure on a surface. Furthermore, the integration points on $S$ and $\Gamma$ may not coincide. This convergence exceeds our expectations, and has triggered our interest for further research.

Currently, we omitted the assessment of the complete error in the three dimensional setting. The reason is that in order to assess this error, it is required that all other errors are negligible. That is, the finite element errors, as well as the quadrature error over $S$ should be negligible, which requires either ultra-fine meshes or higher-order discretization methods. In order to figure this out, we either have to rely on more computational power or on higher order finite elements and quadratures over $S$.

\begin{figure}[h!]
\centering
\subfigure[No intersection between cell and plane]{
\includegraphics[width=0.48\textwidth]{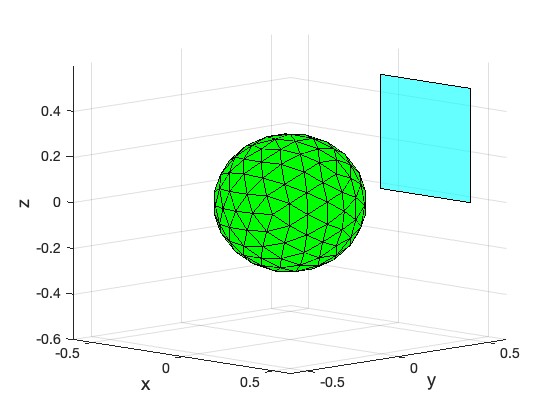}}
\subfigure[Intersection between cell and plane]{ \includegraphics[width=0.48\textwidth]{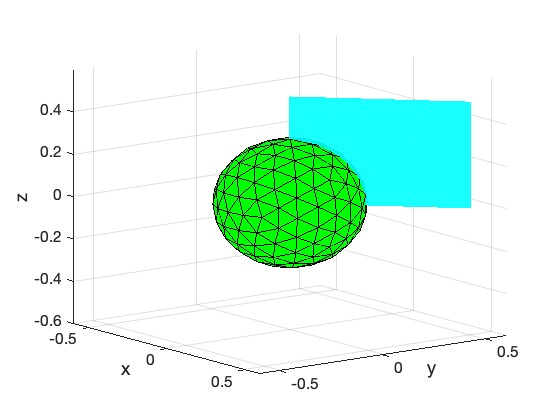}}
\caption{\it Geometry in $\mathbb{R}^3$ with the spherical cell $B({\bf 0},0.3)$ and two regions $S$. (a): $S$ is the parallelogram defined by the points $A(0.5,0.5,0)$, $B(0,0.5,0)$ and $C(0.5,0.5,0.5)$ ($S \cap \Gamma = \emptyset$). (b): $S$ is the parallelogram defined by the points $A(0.5,0.5,0)$, $B(0,0,0)$ and $C(0.5,0.5,0.5)$, which intersects the sphere ($S \cap \Gamma \ne \emptyset$). The dimensions of both configurations are the same.}
\label{fig2}
\end{figure}

\begin{table}\footnotesize
\centering
\caption{\it Weighted ${\bf L}^2$--norm of ${\boldsymbol u}_h^*$ for two parallelograms in the three-dimensional case, namely, a case not intersecting the cell and a case intersecting the cell; see Figure \ref{fig2} as a schematic. The order of convergence is also presented for each case accordingly.}
\label{Table:results_L2_norm_3D_C1}
\begin{tabular}{m{1.5cm}|m{2.5cm}m{1.5cm}||m{2.5cm}m{1.5cm}}
\hline
\textbf{Number of faces on cell boundary} & \textbf{$S \cap \Gamma = \emptyset$} & {\bf Order of convergence } & \textbf{$S \cap \Gamma \ne \emptyset$} & \textbf{Order of convergence} \\ 
\hline
80    & $1.941412\times 10^{-4}$ & - &  $2.890295\times 10^{-4}$  & - \\ 
320   & $2.147246\times 10^{-4}$ & - & $3.195174\times 10^{-4}$ & - \\ 
1280  & $2.203262\times 10^{-4}$ & $1.877579$ &  $3.278651 \times 10^{-4}$ & $1.868786$ \\ 
5120  & $2.217570\times 10^{-4}$ & $1.969053$ &  $3.299971\times 10^{-4}$ & $1.969203$ \\ 
20480 & $2.221166\times 10^{-4}$ & $1.992241$ & $3.305329\times 10^{-4}$ & $1.992274$ \\ 
81920 & $2.222066\times 10^{-4}$ & $1.998059$ & $3.306671\times 10^{-4}$ & $1.998067$\\ 
\hline
\hline
\end{tabular}
\end{table} 

\section{Discussion and Conclusions}
\label{sec:5}
We used the Immersed Interface Method for forces exerted on a curve or surface within two- and three-dimensional media. This method assigns a force density along the curve or surface, incorporating it into the momentum balance through an integral with the Dirac delta distributions. However, this integral cannot always be simplified to an explicit expression. Hence, for the purpose of solving the corresponding boundary value problem with numerical solvers, a quadrature rule is used to approximate the integral over the surface or curve where the forces are applied. This approximation introduces an additional numerical error besides the one from numerical solvers (e.g. finite-element methods or finite-volume methods).

Our investigation on quantifying how the quadrature rule influences the solution started with the linear elasticity equation in two and three dimensions. To better handle the Dirac delta distribution in numerical simulations the singularity removal technique is used, introduced by Gjerde et al. \cite{Gjerde_2019}. We consider an open, non-empty, bounded and connected domain $\Omega$ in which the solution to the linear elasticity equation is composed of the fundamental solution in $\Omega(\mathbb{R}^{d})$ with $d=\{2,3\}$ and an auxiliary term. Due to the singularity in the exact solution that is caused by the Dirac delta distribution, the solution does not lie in the classical Hilbert space ${\bf H}^1$. Hence, following \cite{K_ppl_2014,Bertoluzza_2017}, we are more interested in the convergence rate of the finite element solution in the subdomain of $\Omega$ that \textit{excludes} the part where the Dirac delta distributions are applied.

We demonstrated that the ${\bf L}^2$--error in the solution of the linear elasticity problem in $\Omega$ is of the same order as the error of the quadrature rule that is applied over the curve or surface where the forces are applied. As a computational validation, we utilized the Midpoint Rule to approximate the integral form of the force. We observed that the numerical results agree with our theoretical findings. Two-and three-dimensional cases have been considered.

The ultimate goal of the current study is to show that the quadrature error for the evaluation of the integral over the interface determines the difference between the internal solutions. The results of the model can also be used to subsequently determine the force that is exerted on the outer boundary by computing $\boldsymbol\sigma \cdot \boldsymbol n$ at $\partial \Omega$. A practical application could be a tumor that is growing and pushing on its direct environment. The stresses in the tissue will lead to stresses on the outer boundary (so on different surrounding parts of the body). Furthermore, future work will extend the present study to more complicated modeling frameworks, including viscoelasticity and morpho-poro-viscoelasticity \citep{Sabia2025, Sabiafem2025}, which are more realistic to accurately model biological scenarios for soft tissues. These extensions significantly increase the analytical and computational complexity of the problem. In particular, when fundamental solutions are unavailable, classical analytical techniques such as boundary integral formulations cannot be employed. Moreover, if the governing differential operator is nonlinear, the principle of superposition no longer applies, making it difficult to obtain closed-form solutions and to establish stability and uniqueness. Consequently, future studies will require advanced mathematical analysis and robust numerical methods to address these challenges.

\vskip 2em \noindent
{\bf Declaration of competing interest:}
The authors declare that they have no known competing financial interests or personal relationships that could
have appeared to influence the work reported in this paper.\\[1.5ex]
\noindent{\bf Acknowledgments:} The work of QP was supported by Research England under the Expanding Excellence in England (E3) funding stream, which was awarded to MARS: Mathematics for AI in Real-world Systems in the School of Mathematical Sciences at Lancaster University. The work of EJ is supported in part by the Spanish project PID2022-140108NB-I00 (MCIU/AEI/FEDER, UE), and by the DGA (Grupo de referencia APEDIF, ref. E24$\textunderscore$17R). Further, this work was supported by a fellowship awarded to SA by the Higher Education Commission (HEC) of Pakistan in the framework of project: 1(2)/HRD/OSS-III/BATCH-3/2022/HEC/527.

\begin{appendices}

\section{}\label{Appendix}
Here we present the proof of Lemma \ref{lemma: uniqueness of v_p}. For ease of reading, we reformulate it:
\begin{lemma star}
    The variational problem $\mathrm{(WF_{{\boldsymbol v}_p})}$ has one and only one solution ${\boldsymbol v}_p \in {\bf H}^1(\Omega)$.
    \label{existence}
\end{lemma star}
\begin{proof}
The form $a(.,.)$ is obviously bilinear and symmetric. Furthermore, Korn's Inequality \citep{braess2001finite} assists in proving that $a(.,.)$ is coercive (or positive definite), and the boundedness of the bilinear form follows from symmetry through the Toeplitz-Hellinger Theorem \cite{Edward}. First, we show that there is at most one solution to the variational problem. Suppose that there are two solutions in ${\bf H}^1(\Omega)$ to $\mathrm{(WF_{{\boldsymbol v}_p})}$, then the difference between them, referred to as $\Delta {\boldsymbol v} \in {\bf H}^1_0(\Omega)$, satisfies
\begin{align*}
    a(\Delta {\boldsymbol v}, \phi) = 0, 
    \qquad \forall \phi \in {\bf H}^1_0(\Omega).
\end{align*}
Using $\phi = \Delta {\boldsymbol v}$ and coercivity, it follows that there is $\alpha > 0$ such that 
\begin{align}
\alpha || \Delta \boldsymbol{v} ||^2_{{\bf H}^1(\Omega)} \le a(\Delta {\boldsymbol v}, \Delta {\boldsymbol v}) = 0.
\end{align}
This implies that the difference $\Delta {\boldsymbol v} = 0$, and hence there is at most one solution (note that one can also directly use Lax-Milgram's theorem and substitution of zero).

We establish the existence of ${\boldsymbol v},~{\boldsymbol v}_h \in {\bf H}^1(\Omega)$ in the variational problem \citep{Fichera1973,Lions2005}. From now on, we follow more or less the procedure outlined in Lions and Magenes \cite{lions2012non}. Since ${\boldsymbol u}^*,~{\boldsymbol u}_h^*$ are smooth on $\partial \Omega$, and assuming that $\partial \Omega$ is Lipschitz, it follows from the Trace Extension Theorem \citep{adams2003sobolev} that for both cases (${\boldsymbol v}$ and ${\boldsymbol v}_h$) there is a ${\boldsymbol v}^P = E({\boldsymbol u}^*|_{\partial \Omega}),~
{\boldsymbol v}_h^P = E({\boldsymbol u}_h^*|_{\partial \Omega})\in {{\bf H}^1(\Omega)}$ (in other words there are ${\bf H}^1$--extensions from $\partial \Omega$ to $\Omega$) such that $T({\boldsymbol v}^P) = {\boldsymbol u}^*|_{\partial \Omega}$ and $T({\boldsymbol v}_h^P) = {\boldsymbol u}_h^*|_{\partial \Omega}$ since ${\boldsymbol u}^*|_{\partial \Omega},~{\boldsymbol u}_h^*|_{\partial \Omega} \in {{\bf H}^{1/2}(\partial \Omega)}$, for which $$||{\boldsymbol v}^P||_{{{\bf H}^1(\Omega)}} \le C_I ||{\boldsymbol u}^*||_{{{\bf L}^2(\partial \Omega)}}, \qquad ||{\boldsymbol v}_h^P||_{{{\bf H}^1(\Omega)}} \le C_I^h ||{\boldsymbol u}_h^*||_{{{\bf L}^2(\partial \Omega)}},$$ for a $C_I, C_I^h = C_I,C_I^h(\Omega) > 0$. To this extent, we write ${\boldsymbol v} = {\boldsymbol v}^H + {\boldsymbol v}^P$ and ${\boldsymbol v}_h = {\boldsymbol v}_h^H + {\boldsymbol v}_h^P$, where ${\boldsymbol v}^P |_{\partial \Omega} = -{\boldsymbol u}^*(\boldsymbol{x})$ and ${\boldsymbol v}_h^P |_{\partial \Omega} = -{\boldsymbol u}_h^*(\boldsymbol{x})$. This, respectively, gives
\begin{align*}
\text{Find ${\boldsymbol v}^H \in {{\bf H}^1_0(\Omega)}$ such that }
a({\boldsymbol v}^H,\phi) = -a({\boldsymbol v}^P,\phi), \text{ for all $\phi \in {{\bf H}_0^1(\Omega)}$, \text{ and }}
\end{align*}
\begin{align*}
\text{Find ${\boldsymbol v}_h^H \in {{\bf H}^1_0(\Omega)}$ such that }
a({\boldsymbol v}_h^H,\phi) = -a({\boldsymbol v}_h^P,\phi), \text{ for all $\phi \in {{\bf H}_0^1(\Omega)}$.}
\end{align*}
By substitution and some algebraic manipulations, it follows that the bilinear form is bounded in ${\bf H}^1_0(\Omega)$. The right-hand side is then also bounded in $\phi \in {{\bf H}^1_0(\Omega)}$. Together with coerciveness, it follows from Lax-Milgram's Lemma that there exists ${\boldsymbol v}^H, {\boldsymbol v}_h^H \in {\bf H}_0^1(\Omega)$ that solve the above variational problems. Since we observed that ${\boldsymbol v}^P$ and ${\boldsymbol v}_h^P$ exist in ${\bf H}^1(\Omega)$, it follows that ${\boldsymbol v} = {\boldsymbol v}^P + {\boldsymbol v}^H$ and ${\boldsymbol v}_h = {\boldsymbol v}_h^P + {\boldsymbol v}_h^H$ exist in ${\bf H}^1(\Omega)$. We also saw earlier that there is at most one respective ${\boldsymbol v}, ~{\boldsymbol v}_h \in{\bf H}^1(\Omega)$. This concludes the proof.
\end{proof}

\end{appendices}

\bibliography{sn-bibliography}

\end{document}